\documentclass[11pt,reqno]{amsart}
\usepackage{hyperref,multirow}
\usepackage[all]{xy}

\title{Some non-special cubic fourfolds}

\author[N.~Addington]{Nicolas Addington}
\address{Nicolas Addington \\
Department of Mathematics \\
University of Oregon \\
Eugene, Oregon 97403 \\
United States}
\email{adding@uoregon.edu}

\author[A.~Auel]{Asher Auel}
\address{Asher Auel \\
Department of Mathematics\\
Yale University\\
New Haven, Connecticut 06511\\
United States}
\email{asher.auel@yale.edu}

\newcommand \C {\mathbb C}
\newcommand \F {\mathbb F}
\renewcommand \P {\mathbb P}
\newcommand \Q {\mathbb Q}
\newcommand \Z {\mathbb Z}

\newcommand \prim {\textup{prim}}
\newcommand \Vap {\textit{V-ap}}
\newcommand \IR {\textit{IR}}
\newcommand \rkthree {\textit{rk3}}
\newcommand \copl {\textit{copl}}
\newcommand \an {\textup{an}}
\newcommand \et {\textup{\'et}}
\newcommand \sing {\textup{sing}}

\DeclareMathOperator \Sym {Sym}

\DeclareMathOperator \CH {CH}
\DeclareMathOperator \cl {cl}
\DeclareMathOperator \tr {tr}
\DeclareMathOperator \Bl {Bl}

\newtheorem {thm} {Theorem}
\newtheorem {prop} {Proposition} [section]

\makeatletter
\hypersetup{
  pdftitle={\@title},
  pdfauthor={Nicolas Addington and Asher Auel},
  colorlinks=true,linkcolor=blue,anchorcolor=blue,citecolor=blue
}
\makeatother

\begin{document}

\begin{abstract}
In \cite{rv}, Ranestad and Voisin showed, quite surprisingly, that the
divisor in the moduli space of cubic fourfolds consisting of cubics
``apolar to a Veronese surface'' is not a Noether--Lefschetz divisor.
We give an independent proof of this by exhibiting an explicit 
cubic fourfold $X$ in the divisor and using point counting methods over
finite fields to show $X$ is Noether--Lefschetz general.
We also show that two other divisors considered in \cite{rv} are not
Noether--Lefschetz divisors.
\end{abstract}

\maketitle

\section{Introduction}

In \cite{rv}, Ranestad and Voisin introduced some new divisors in the
moduli space of smooth complex cubic fourfolds, quite different from
Hassett's Noether--Lefschetz divisors \cite{hassett}.  A cubic $X
\subset \P^5$ is called \emph{special} if
\[ H^{2,2}_\prim(X,\Z) := H^4_\prim(X,\Z) \cap H^{2,2}(X) \]
is non-zero, or equivalently if $X$ contains a surface not homologous to
a complete intersection. The locus of special cubic fourfolds is a
countable union of irreducible divisors in the moduli space, called
Noether--Lefschetz divisors.  Special cubic fourfolds often have rich
connections to K3 surfaces, and it is expected that all rational cubic
fourfolds are special; see \cite{hassett_survey} for a recent survey
of the topic.

Ranestad and Voisin's divisors are constructed in a much more
algebraic way, using apolarity.  Briefly, a cubic fourfold $X$ cut
out by a polynomial $f(y_0,\dotsc,y_5)$ is said to be \emph{apolar} to
an ideal generated by quadrics,
\[ I = \langle q_1, \dotsc, q_m \rangle \subset \C[y_0,\dotsc,y_5], \]
if, writing $q_i = \sum a_{ijk}\, y_j y_k$, we have
\[ \sum a_{ijk}\, \partial_j \partial_k f = 0 \text{ for all }i. \]
Ranestad and Voisin showed that the following loci are irreducible divisors in the
moduli space of cubic fourfolds: $D_\Vap$, the set of cubics
apolar to a Veronese surface; $D_\IR$, the set of cubics
apolar to a quartic scroll; and $D_\rkthree$, the closure of the set
of cubics apolar to the union of a plane and a disjoint
hyperplane.  They showed that $D_\Vap$ is
\emph{not} a Noether--Lefschetz divisor, by carefully analyzing its
singularities.  From this they deduced that for a generic cubic $X$,
the ``varieties of sums of powers'' of the polynomial $f$, which is a
hyperk\"ahler fourfold, is not Hodge-theoretically related to
the Fano variety of lines on $X$, a better-known hyperk\"ahler
fourfold.  They remarked that $D_\rkthree$ is
``presumably'' not a Noether--Lefschetz divisor, and that if one
could prove that $D_\IR$ is not a Noether--Lefschetz divisor then it
would give another approach to proving their main theorem.

We were very surprised to learn that $D_\Vap$ is not a
Noether--Lefschetz divisor: we would have guessed that it was
Hassett's divisor $\mathcal C_{38}$, for the following reason.  Cubic
fourfolds in $\mathcal C_{38}$, which are conjectured to be rational,
have associated K3 surfaces of degree 38.  Mukai \cite{mukai} observed
that the generic such K3 surface $S$ can be described as the variety of
sums of powers of a plane sextic $g(x_0, x_1, x_2)$; see
\cite[Thm.~1.7(iii)]{rs} for a more detailed account.  A natural way to construct a
cubic fourfold from $g$ is to consider the multiplication map
\[ m\colon \Sym^3 \Sym^2 \C^3 \to \Sym^6 \C^3 \]
and its transpose
\[ m^\vee\colon \Sym^6 \C^{3\vee} \to \Sym^3 \Sym^2 \C^{3\vee}. \]
Then $m^\vee(g)$ cuts out a cubic $X \subset \P(\Sym^2 \C^{3\vee}) =
\P^5$, typically smooth.  By \cite[Lem.~1.7]{rv}, the cubics obtained
this way are exactly those in $D_\Vap$.  Though it seemed reasonable
to expect that the cubic $X$ would be Hodge-theoretically associated
with the K3 surface $S$, Ranestad and Voisin's result implies that it
cannot be.

Since the result is so surprising, and the proof quite difficult, at
least to our eyes, we thought it worthwhile to seek experimental
confirmation.  In this note, we give a computer-aided proof of the
following result, and in particular a more direct proof of Ranestad
and Voisin's result:

\begin{thm} \label{thm1} 
There is an explicit sextic polynomial $g$, defined over $\Q$,
such that the cubic fourfold $X$ cut out by $m^\vee(g)$ is smooth and satisfies
$H^{2,2}_\prim(X,\Z)=0$.  In particular, $X \in D_\Vap$, but $X$ is
not in any Noether--Lefschetz divisor.
\end{thm}

We also confirm Ranestad and Voisin's expectations for the other two divisors mentioned above:\footnote{They also studied a fourth
divisor $D_\copl$, not defined using apolarity, but we
were unable to find a suitable cubic in that divisor using the
technique described below.  Probably one could be found by working
modulo 5, but that would forfeit many of the computational advantages
of working modulo 2.}

\begin{thm} \label{thm2} 
There is an explicit cubic fourfold $X \in
D_\IR$, defined over $\Q$, 
with $H^{2,2}_\prim(X,\Z)=0$.  In
particular, $D_\IR$ is not a Noether--Lefschetz divisor.
\end{thm}

\begin{thm} \label{thm3} 
There is an explicit cubic fourfold $X \in
D_\rkthree$, defined over~$\Q$, 
with $H^{2,2}_\prim(X,\Z)=0$.  In particular, $D_\rkthree$ is not a
Noether--Lefschetz divisor.
\end{thm}

\noindent Thus it seems that apolarity tends to produce cubic
fourfolds of a different character than those considered by Hassett.  It would be very interesting to know if there
is any connection with rationality.

We follow a strategy developed by van Luijk~\cite{vl} and refined by
Elsenhans and Jahnel~\cite{ej_deg2, ej_duke}, for producing explicit
K3 surfaces of Picard rank~1.  We find an explicit cubic fourfold with
good reduction modulo 2, then count points over $\F_{2^m}$ for
$m=1,2,\dotsc,11$ to determine the eigenvalues of Frobenius acting on
$H^4_\prim(X_{\overline\F_2}, \Q_\ell(2))$, which give a bound on the
rank of $H^{2,2}_\prim(X,\Z)$.  In \S\ref{sec:van_Luijk}, we give the
details of adapting van Luijk's method to cubic fourfolds.

On the one hand, our task is simpler than van Luijk's: since the
geometric Picard rank of a K3 surface over a finite field is
necessarily even, to show that a K3 surface has Picard rank 1, van
Luijk had to work modulo two different primes and compare intersection
forms; but here we need only work modulo one prime.  On the other hand, a
fourfold is much bigger than a surface, and it is infeasible to count
points naively by iterating over $\P^5$.  Nor can we control the
cohomology of $X$ by counting points on an associated K3 surface as in
\cite{abbv} or \cite{hvv}, since there is none.  In \S\ref{sec:alg} we
explain how to exploit the conic bundle structure on the blow-up of
$X$ along a line, so that to count points we only need to iterate over
$\P^3$, and with a little more work, only over $\P^2$.  The same idea
was used to count points on cubic \emph{three}folds by Debarre,
Laface, and Roulleau \cite[\S4.3]{dlr}, who trace it back to Bombieri
and Swinnerton--Dyer \cite{bsd}.  Whereas those papers restrict to odd
characteristic, we find that the hassle of working with conics in
characteristic 2 is more than repaid by the fact that computation in
$\F_{2^m}$ is so fast.

We do not use the $p$-adic cohomology methods of Kedlaya, Harvey, and others \cite{akr,harvey,chk}.  While these methods are surely the way of the future, they are much harder to implement than our algorithm, and the available implementations are not quite ready to handle cubic fourfolds.

In \S\ref{sec:equations}, we give the explicit polynomials and
the point counts needed to prove Theorems~\ref{thm1},
\ref{thm2}, and \ref{thm3}.  In \S\ref{sec:checks}, we conclude
with some remarks about computer implementation and verification.

The existence of
Noether--Lefschetz general cubic fourfolds (and other complete intersections)
defined over $\Q$ was first proved by Terasoma~\cite{terasoma}, although his proof is
not constructive.  Elsenhans and Jahnel gave an explicit example in
\cite[Example~3.15]{ej_duke}, also using point-counting methods.
But the existence of Noether--Lefschetz general cubic fourfolds
with specified algebraic properties is far from clear \emph{a priori}.

\subsection*{Acknowledgements}
We thank J.-L.~Colliot-Th\'el\`ene, E.~Costa, D.~Harvey, B.~Hassett,
A.~Kuznetsov, K.~Ranestad, R.P.~Thomas, A.~V\'arilly-Alvarado, and B.~Viray for
helpful conversations, and the organizers of the Fall 2016 AGNES
workshop at UMass Amherst for their hospitality.  In the course of
this project we used the computer algebra system Macaulay2 \cite{M2}
extensively, and Magma \cite{magma} to a lesser extent.  The second
author was partially supported by NSA Young Investigator Grant
H98230-16-1-0321.

\section{Adaptation of van Luijk's method}
\label{sec:van_Luijk}

In this section we adapt the method developed in \cite{vl} from
K3 surfaces to cubic fourfolds.  We begin with the following
proposition, which is similar to \cite[Cor.~6.3]{vl_older}.  Note that
due to our choice of Tate twist, our Frobenius eigenvalues have
absolute value 1 rather than $q^i$.  

\begin{prop}
\label{prop}
Let $R$ be a discrete valuation ring of a number field $L$ with
residue field $k \cong \F_q$ for $q = p^r$, and let $X$ be a smooth
projective scheme over $R$.  Let $X^\an$ denote
the complex manifold associated to the complex variety~$X_\C$.  Let
$\Phi\colon X_k \to X_k$ be the $r$-th power absolute Frobenius, let
$\ell$ be a prime different from $p$, and let $\Phi^*$ be the
automorphism of
\[ H^{2i}_\et(X_{\bar k}, \Q_\ell(i)) \]
induced by $\Phi \times 1$ on $X_k \times \bar k$.

Then the rank of the image of the cycle class map
\begin{equation} \label{cycle_class_map}
\CH^i(X_\C) \xrightarrow{\ \cl\ } H^{2i}(X^\an, \Z(i))
\end{equation}
is less than or equal to the number of eigenvalues of $\Phi^*$, counted with multiplicity, that are roots of unity.

In particular, if the Hodge conjecture holds for codimension-$i$ cycles on $X$, then the rank of $H^{2i}(X^\an,\Z) \cap H^{i,i}(X^\an)$ 
is bounded above by the number of such eigenvalues.
\end{prop}

\begin{proof}
The rank of the image of \eqref{cycle_class_map} agrees with the rank of the image of
\[ \CH^i(X_\C) \xrightarrow{\ \cl\ } H^{2i}(X^\an,\Z_\ell(i)). \]
By the comparison theorem between singular and $\ell$-adic cohomology, this agrees with the rank of the image of
\[ \CH^i(X_\C) \xrightarrow{\ \cl\ } H^{2i}_\et(X_\C,\Z_\ell(i)). \]
Now let $K$ be the field of fractions of the completion $\widehat R$, and consider the commutative diagram\footnote{Alternatively we could have embedded $\bar K \hookrightarrow \C$, but we preferred to use the more natural embeddings $\C \hookleftarrow \bar L \hookrightarrow \bar K$.}
\[ \xymatrix{
\CH^i(X_\C) \ar[r]^-\cl & H^{2i}_\et(X_\C,\Z_\ell(i)) \\
\CH^i(X_{\bar L}) \ar[u] \ar[d] \ar[r]^-\cl & H^{2i}_\et(X_{\bar L},\Z_\ell(i)) \ar[u]_\cong \ar[d]^\cong \\
\CH^i(X_{\bar K}) \ar[r]^-\cl & H^{2i}_\et(X_{\bar K},\Z_\ell(i)).
} \]
The right-hand vertical maps are isomorphisms by smooth base change, and while the left-hand vertical maps are typically not isomorphisms, the images of the three horizontal maps agree thanks to the existence of Hilbert schemes, as remarked in \cite[Rem.~46]{cs}.

Next we have a commutative square
\[ \xymatrix{
\CH^i(X_{\bar K}) \ar[d]^\sigma \ar[r]^-\cl & H^{2i}_\et(X_{\bar K},\Z_\ell(i)) \ar[d]^\cong \\
\CH^i(X_{\bar k}) \ar[r]^-\cl & H^{2i}_\et(X_{\bar k},\Z_\ell(i)),
} \]
where the left-hand vertical map is the specialization map for Chow
groups; see Fulton~\cite[Example~20.3.5]{fulton} for the commutativity of the square.  Thus the rank of the image of
the top horizontal map is less than or equal to that of the bottom
one.

Finally we consider the cycle class map after tensoring with $\Q_\ell$
\[ 
\CH^i(X_{\bar k}) \otimes \Q_\ell \xrightarrow{\ \cl\ }
H^{2i}(X_{\bar k},\Q_\ell(i)) 
\]
and recall that cycles on $X_{\bar k}$ are defined over some finite
extension of $k$, hence are fixed by some power of Frobenius, hence
their classes in cohomology are eigenvectors with eigenvalues a root
of unity as in the proof of \cite[Cor.~6.3]{vl_older}.
\end{proof}

In our
application, we will take $R = \Z_{(2)}$, so $L = \Q$, $q = p = 2$, and $K=\Q_2$.

\medskip

Now specialize to the case where $X$ is a cubic fourfold.  The Hodge
conjecture holds for cubic fourfolds
\cite{zucker, murre, voisin_Z_hodge}, so to show that
$H^{2,2}_\prim(X, \Z) = 0$ it is enough to show that no eigenvalue of
$\Phi^*$ acting on
\[ V := H^4_{\et,\prim}(X_{\bar k}, \Q_\ell(2)) \cong \Q_\ell^{22} \]
is a root of unity, or equivalently that the characteristic polynomial
\[ \chi(t) := \det(t \cdot \mathrm{Id}_V - \Phi^*|_V) \]
has no cyclotomic factor.  For this it is enough to show that $\chi$ is irreducible over $\Q$ and that not all its coefficients are integers.

The cohomology of $X$ is
\[ 
H^i_\et(X_{\bar k}, \Q_\ell(i)) = \begin{cases}
\Q_\ell & i=0, \\
\Q_\ell \cdot h & i=2 \\
\Q_\ell \cdot h^2 \, \oplus \, V & i=4 \\
\Q_\ell \cdot h^3 & i=6 \\
\Q_\ell \cdot h^4 & i=8 \\
0 &\text{otherwise,}
\end{cases} \]
where $h$ is the hyperplane class, so by the Lefschetz trace formula we have
\begin{equation} \label{point_count}
\#X(\F_{q^m}) = 1 + q^m + q^{2m}\Bigl(1+ \tr(\Phi^{*m}|_V) \Bigr) + q^{3m} + q^{4m}.
\end{equation}
The method of passing from traces of powers of $\Phi^*|_V$ to the
characteristic polynomial using Newton's identities is discussed in
\cite[\S3]{vl}, \cite[\S3]{ej_deg2}, or \cite[\S6.1]{hvv}.  Thanks to
the functional equation $\chi(t) = \pm t^{22} \chi(t^{-1})$ it is
usually enough to count up to $m=11$.

\section{The algorithm using conic bundles} 
\label{sec:alg} 
How then can we compute the point counts \eqref{point_count} for an
explicit cubic with $q=2$ and $m = 1, 2, \dotsc, 11$?  As we said in
the introduction, it is not feasible to iterate over $\P^5(\F_{2^m})$,
evaluating our cubic polynomial at every point: in Magma this would
take many years, and in a program written optimized specially for the
purpose it would take months, or at best weeks.  Instead we project
from a line to obtain a conic fibration.

Continue to work with a smooth cubic $X$ defined over an arbitrary $\F_q$.  Choose a line $l \subset X$ defined over $\F_q$; by \cite{dlr} such a line always exists for $q = 2$ or $q \ge 5$, and probably for $q=3$ or 4 as well.  Change variables so that $l$ is given by $y_0 = y_1 = y_2 = y_3 = 0$.  Then we can write the equation of $X$ as
\[ A y_4^2 + B y_4 y_5 + C y_5^2 + D y_4 + E y_5 + F, \]
where $A$, $B$, and $C$ are linear in $y_0, \dotsc, y_3$, $C$ and $D$
are quadratic, and $F$ is cubic.  If $A, \dotsc, F$ vanish
simultaneously at some point of $\P^3$ then $X$ contains a plane,
contributing an unwanted Frobenius eigenvalue, so we stop.  Otherwise
we obtain a flat conic bundle
\[ \Bl_l(X) \longrightarrow \P^3_{(y_0:\dotsc:y_3)} \]
with fibers given by the homogenization of the quadratic form above.
Now we use the following.
\begin{prop}
Let $Z$ be an $\F_q$-scheme of finite type, let $\pi\colon Y \to Z$ be
a flat conic bundle, let $\Delta \subset Z$ be the locus
parametrizing degenerate conics, and let $\tilde \Delta$ be the
(possibly branched) double cover of $\Delta$ parametrizing lines in
the fibers of $\pi$.  Then
\begin{equation} \label{conic_bundle_count}
\#Y(\F_q) = (q+1) \cdot \#Z\ +\ q \cdot (\#\tilde\Delta - \#\Delta).
\end{equation}
\end{prop}
\begin{proof}
A smooth conic over $\F_q$ is isomorphic to $\P^1$, hence has $q+1$ points.  For a singular conic, there are three possibilities:
\setlength \leftmargini {1.5em}
\begin{itemize}
\item a pair of lines defined over $\F_q$, contributing $2q+1$ points;
\item a pair of conjugate lines defined over $\F_{q^2}$, contributing only one $\F_q$-point;
\item a double line, contributing $q+1$ points.
\end{itemize}
The fiber of $\tilde \Delta$ over the relevant point of $\Delta$ consists of 2, 0, or 1 points respectively.  Thus we have
\[ \#Y(\F_q) =
\underbrace{(q+1) \cdot (\#Z - \#\Delta)}_\text{from smooth conics} +
\underbrace{(q\cdot \#\tilde\Delta + \#\Delta)}_\text{from singular conics}, \]
which simplifies to give \eqref{conic_bundle_count}.
\end{proof}

In our case, with $Y = \Bl_l(X)$ and $Z = \P^3$, this yields
\[ \#X(\F_q) = q^4 + q^3 + q(\#\tilde\Delta - \#\Delta) + q + 1. \]
The discriminant locus $\Delta \subset \P^3$ is cut out by the quintic polynomial
\begin{equation} \label{discr}
A E^2 + B^2 F + C D^2 - BDE - 4ACF.
\end{equation}
This formula remains valid in characteristic 2, although of course the last
term vanishes.  The double cover $\tilde\Delta$ can also be described as
the variety of lines on $X$ that meet $l$.\footnote{The topology of $\tilde\Delta$ over $\C$ has been studied in \cite[\S3,
Lemmas~1--3]{voisin_thesis}.  For a generic $l \subset X$, it is a
smooth surface with Hodge diamond

\centering \tiny
\setlength \tabcolsep {2pt}
\begin{tabular}{ccccc}
& & 1 \\
& 0 & & 0 \\
5 & & 50 & & 5. \\
& 0 & & 0 \\
& & 1 \\[1ex]
\end{tabular}
}

So we can iterate over $\P^3$ and count points on $\Delta$ and $\tilde
\Delta$.  To count points on $\tilde \Delta$ in characteristic 2, we
note that if $B = D = E = 0$ then the conic is a double line;
otherwise we compute an Arf invariant: if $B \ne 0$ (resp.\ $D \ne 0$
or $E \ne 0$), then the conic has $2q+1$ points if $AC/B^2$ (resp.\
$AF/D^2$ or $CF/E^2$) is of the form $a^2+a$ for some $a \in \F_q$,
and 1 point if it is not.

This algorithm runs up to $q = 2^{11}$ in about half a minute on the first author's laptop.  But to find the explicit cubics below we had to  search through dozens of candidates, so it was worthwhile to make a further optimization, iterating only over $\Delta$ rather than all of $\P^3$, as follows.

The quintic $\Delta$ is not smooth; in characteristic 2, it is singular
at least along the locus where
\[ B = D = E = 0, \]
which has expected dimension 0 and degree 4.  Suppose this locus contains
an $\F_2$-point $y$.\footnote{In practice this usually happens, although not
always.  That is, there exist smooth cubics $X$ and
$\F_2$-lines $l \subset X$ such that $\Delta_\sing$ has no $\F_2$-point,
but they are relatively rare.  We have not encountered a cubic $X$
such that for \emph{every} $\F_2$-line $l \subset X$, $\Delta_\sing$
has no $\F_2$-point.  We wonder whether any such cubic exists.}
Projecting from $y$, the quintic $\Delta$ becomes a 3-to-1 cover of
$\P^2$, so we can iterate over $\P^2$ and find the three (or fewer)
sheets of the cover at each point with a suitable version of Cardano's
formula \cite[Exercise~14.7.15]{df}.

With this improvement the algorithm runs up to $q = 2^{11}$ in less
than a second, and up to $q = 2^{14}$ in a little more than a minute.
In \S\ref{sec:checks} we make some practical comments about our implementation
of the algorithm, and sanity checks on the output.

\section{The explicit cubics}
\label{sec:equations}

\subsection{Proof of Theorem 1} \ 
\label{sec:thm1}

Let us begin by discussing the map $m^\vee$ from the introduction in very concrete terms, embracing the monomial basis for the polynomial ring rather than working invariantly, and staying in characteristic 0 as long as possible to avoid discussing divided powers.

Let $R = \C[x_0, \dotsc, x_n]$, and let $R_d \subset R$ be the
subspace of homogeneous polynomials of degree $d$.  We identify $R_1$
with its dual via the pairing
\[ \langle x_i, x_j \rangle = \frac{\partial}{\partial x_i} x_j = \delta_{ij}, \]
and extend this to a pairing
\[ R_k \otimes R_d \to R_{d-k} \]
for positive integers $k \le d$, again by differentiation.  If $k=d$ this is a perfect, symmetric pairing.  We have, for example,
\[ \langle x_0 x_1,\, x_0 x_1 \rangle = \tfrac\partial{\partial x_0} \tfrac\partial{\partial x_1} x_0 x_1 = 1, \]
but
\[ \langle x_0^2,\, x_0^2 \rangle = \tfrac\partial{\partial x_0} \tfrac\partial{\partial x_0} x_0^2 = 2, \]
so the monomials form an orthogonal basis for $R_d$ but not an orthonormal basis.  For $k > d$ we set $\langle R_k, R_d \rangle = 0$.

Now with a view toward Theorem \ref{thm1}, let $R = \C[x_0,x_1,x_2]$ and $S = \C[y_0, \dotsc, y_5]$.  The isomorphism
\[ m\colon S_1 \to R_2 \]
given by
\begin{align*}
y_0 &\mapsto x_0^2 &
y_1 &\mapsto x_0 x_1 &
y_2 &\mapsto x_0 x_2 \\
y_3 &\mapsto x_1^2 &
y_4 &\mapsto x_1 x_2 &
y_5 &\mapsto x_2^2.
\end{align*}
induces a map
\[ m\colon S_d \to R_{2d} \]
for all $d$.

Let $g \in R_6$ be given by
\begin{align*}
g &= \frac1{30} x_0^5 x_1
+ \frac16 x_0^4 x_1^2
+ \frac16 x_0^2 x_1^4
+ \frac1{30} x_0 x_1^5
+ \frac1{120} x_1^6 \\
&+ \frac43 x_0^3 x_1^2 x_2
+ \frac23 x_0^2 x_1^3 x_2
+ \frac16 x_0^4 x_2^2
+ 2 x_0^2 x_1^2 x_2^2
+ \frac13 x_0 x_1^3 x_2^2 \\
&+ \frac1{12} x_1^4 x_2^2
+ \frac23 x_0^2 x_1 x_2^3
+ \frac16 x_1^3 x_2^3
+ \frac13 x_0^2 x_2^4
+ \frac1{15} x_1 x_2^5,
\end{align*}
and let $f \in S_3$ be given by
\begin{align*}
f &= 2 y_0^2 y_1
+ 4 y_0 y_1^2
+ 8 y_1^2 y_2
+ 4 y_0 y_2^2
+ 4 y_0^2 y_3
+ 4 y_1^2 y_3 \\
&+ 16 y_0 y_2 y_3
+ 8 y_1 y_2 y_3
+ 8 y_2^2 y_3
+ 4 y_0 y_3^2
+ 2 y_1 y_3^2
+ y_3^3 \\
&+ 16 y_0 y_1 y_4
+ 4 y_1^2 y_4
+ 16 y_1 y_2 y_4
+ 4 y_2^2 y_4
+ 8 y_0 y_3 y_4 \\
&+ 4 y_2 y_3 y_4
+ 8 y_0 y_4^2
+ 2 y_1 y_4^2
+ 2 y_3 y_4^2
+ y_4^3
+ 4 y_0^2 y_5 \\
&+ 8 y_1^2 y_5
+ 8 y_1 y_2 y_5
+ 8 y_2^2 y_5
+ 16 y_0 y_3 y_5
+ 4 y_1 y_3 y_5 \\
&+ 2 y_3^2 y_5
+ 8 y_0 y_4 y_5
+ 6 y_3 y_4 y_5
+ 8 y_0 y_5^2
+ 4 y_4 y_5^2.
\end{align*}
We claim that $f = m^\vee(g)$, i.e.\ that
\[ \langle h, f \rangle = \langle m(h), g \rangle \]
for all $h \in S_3$.  This can be checked tediously by hand, or with the Macaulay2 code given in the ancillary file \verb|thm1.m2|.

Let $X \subset \P^5$ be the hypersurface cut out by $f$.  After substituting
\begin{align*}
y_1 &\mapsto \tfrac12 y_1, &
y_2 &\mapsto \tfrac12 y_2, &
y_5 \mapsto \tfrac12 y_5,
\end{align*}
we obtain a model of $X$ with good reduction modulo 2.  Its reduction contains the line
\[ y_0 + y_3 = y_1 = y_2 + y_3 = y_4 = 0. \] 
The point counts of $X$ over $\F_{2^m}$ are given in Table~\ref{point_counts}.  Thus the characteristic polynomial of $\Phi^*$ acting on $H^4_{\et,\prim}(X_{\bar k}, \Q_\ell(2))$
is
\begin{multline*}
\chi(t) = t^{22} - \frac32 t^{20} + \frac32 t^{18} - t^{16} + \frac12 t^{15} + \frac12 t^{14} - t^{13} + \frac32 t^{11} \\
- t^9 + \frac12 t^8 + \frac12 t^7 - t^6 + \frac32 t^4 - \frac32 t^2 + 1,
\end{multline*}
which is irreducible over $\Q$.  By our discussion in \S\ref{sec:van_Luijk}, this proves Theorem~\ref{thm1}.

\subsection{Proof of Theorem 2} \ 
\label{sec:thm2}

Continue to let $S = \C[y_0, \dotsc, y_5]$.  A homogeneous polynomial $f \in S$ is said to be \emph{apolar} to a homogeneous ideal $I \subset S$ if
\[ \langle i, f \rangle = 0 \qquad \text{for all } i \in I. \]
It is enough to check this on a set of generators for $I$.

Ranestad and Voisin observe \cite[Lem.~1.7]{rv} that a cubic is in the image of $m^\vee\colon R_6 \to S_3$ if and only if it is apolar to ideal generated by the $2 \times 2$ minors of
\[ \begin{pmatrix} y_0 & y_1 & y_2 \\ y_1 & y_3 & y_4 \\ y_2 & y_4 & y_5 \end{pmatrix}, \]
which cuts out a Veronese surface.  This is checked for the previous section's cubic in \verb|thm1.m2|.

For Theorem \ref{thm2}, we take
\begin{align*} f ={}& y_0^3
+ 2 y_1 y_2^2
+ y_2^3
+ y_0^2 y_3
+ 2 y_0 y_1 y_3
+ 8 y_1 y_2 y_3
+ y_0^2 y_4
+ 4 y_1^2 y_4 \\
&+ 8 y_0 y_2 y_4
+ y_2^2 y_4
+ 4 y_2 y_3 y_4
+ y_3^2 y_4
+ 2 y_1 y_4^2
+ y_2 y_4^2
+ y_4^3 \\
&+ 8 y_0 y_1 y_5
+ 2 y_1 y_2 y_5
+ 4 y_1 y_3 y_5
+ 2 y_2 y_3 y_5
+ 4 y_0 y_4 y_5
+ 2 y_1 y_4 y_5 \\
&+ 6 y_3 y_4 y_5
+ y_4^2 y_5
+ y_0 y_5^2
+ y_2 y_5^2
+ y_3 y_5^2
+ y_4 y_5^2.
\end{align*}
This is apolar to the ideal generated by the $2 \times 2$ minors of the matrix
\[ \begin{pmatrix}
y_0 & y_1 & y_3 & y_4 \\
y_1 & y_2 & y_4 & y_5
\end{pmatrix}, \]
which cuts out a quartic scroll.  Apolarity can be checked by hand or with \verb|thm2.m2|.\footnote{Ranestad and Voisin gave a different definition of $D_\IR$ and proved that a cubic of Waring rank 10 (the maximum possible) is in $D_\IR$ if and only if it is apolar to a quartic scroll \cite[Lem.~2.4]{rv}.  Our cubic does have rank 10, as can be checked using \cite[Lem.~3.18]{rv}.  But in fact the rank condition can be ignored: the cubic forms that are apolar to a given quartic scroll form a linear space, in which the general one has rank 10, so $D_\IR$ consists of all cubics apolar to a quartic scroll, with no restriction on rank.  We thank K.~Ranestad for explaining this to us.}

Let $X \subset \P^5$ be the hypersurface cut out by $f$.  After substituting $y_1 \mapsto \frac12 y_1$ we obtain a model of $X$ with good reduction modulo 2.  It contains the line
\[ y_0 = y_2 = y_3 = y_4 = 0. \] 
The point counts of $X$ over $\F_{2^m}$ are given in Table~\ref{point_counts}.  Thus the characteristic polynomial of $\Phi^*$ acting on $H^4_{\et,\prim}(X_{\bar k}, \Q_\ell(2))$
is
\begin{multline*}
\chi(t) = t^{22} + t^{20} + \frac12 t^{19} + \frac12 t^{18} + \frac12 t^{17} - \frac12 t^{14} - \frac12 t^{13} - \frac32 t^{12} - \frac12 t^{11} \\
- \frac32 t^{10} - \frac12 t^9 - \frac12 t^8 + \frac12 t^5 + \frac12 t^4 + \frac12 t^3  + t^2 + 1,
\end{multline*}
which is irreducible over $\Q$.  By our discussion in \S\ref{sec:van_Luijk}, this proves Theorem~\ref{thm2}.

\subsection{Proof of Theorem 3} \ 
\label{sec:thm3}

The cubic fourfold $X$ cut out by
\begin{align*} f ={}& y_0^2 y_1
+ y_0^2 y_2
+ y_0 y_1 y_2
+ y_1 y_2^2
+ y_2^3
+ y_1^2 y_3
+ y_0 y_2 y_3
+ y_0 y_3^2 \\
&+ y_1 y_3^2
+ y_0 y_1 y_4
+ y_0 y_2 y_4
+ y_1 y_2 y_4
+ y_2^2 y_4
+ y_0 y_3 y_4
+ y_1 y_3 y_4 \\
&+ y_2 y_3 y_4
+ y_0 y_4^2
+ y_1 y_4^2
+ y_4^3
+ y_3^2 y_5
+ y_3 y_4 y_5
+ y_4^2 y_5
+ y_4 y_5^2
+ y_5^3
\end{align*}
has good reduction modulo 2.  The polynomial $f$ is apolar to the ideal
\[ \langle y_0 y_5,\, y_1 y_5,\, y_2 y_5 \rangle, \]
as can be checked by hand or with \verb|thm3.m2|.  We do not review the definition of $D_\rkthree$, but only refer to the proof of \cite[Lem.~2.1]{rv} for the fact that this implies $X \in D_\rkthree$.

The reduction of $X$ contains the line
\[ y_0 = y_1 + y_3 = y_2 = y_4 + y_5 = 0. \]
The point counts of $X$ over $\F_{2^m}$ are given in
Table~\ref{point_counts}.  Thus the characteristic polynomial of $\Phi^*$ acting on $H^4_{\et,\prim}(X_{\bar k}, \Q_\ell(2))$
is
\begin{multline*}
\chi(t) = t^{22} - \frac12 t^{21} + \frac32 t^{20} - \frac12 t^{19} - \frac32 t^{16} + \frac12 t^{15} - t^{14} + \frac12 t^{13} + \frac12 t^{12} + \frac12 t^{11} \\
+ \frac12 t^{10} + \frac12 t^9 - t^8 + \frac12 t^7 - \frac32 t^6 - \frac12 t^3 + \frac32 t^2 - \frac12 t + 1.
\end{multline*}
which is irreducible over $\Q$.  By our discussion in \S\ref{sec:van_Luijk}, this proves Theorem~\ref{thm3}.

\begin{table}
\[ \begin{array}{c|r|r|r}
\multirow{2}{*}{m} & \multicolumn{3}{c}{\#X(\F_{2^m})} \\
\cline{2-4}
& \hfill \text{Theorem \ref{thm1}} \hfill & \hfill \text{Theorem \ref{thm2}} \hfill & \hfill \text{Theorem \ref{thm3}} \hfill \\
\hline
1 & 31 & 31 & 33 \\
2 & 389 & 309 & 297 \\
3 & 4\,681 & 4\,585 & 4\,641 \\
4 & 69\,521 & 69\,905 & 70\, 945\\
5 & 1\,082\,401 & 1\,082\,401 & 1\,084\,033 \\
6 & 17\,040\,449 & 17\,050\,689 & 17\,057\,409 \\
7 & 270\,491\,777 & 270\,577\,793 & 270\,525\,953 \\
8 & 4\,311\,818\,497 & 4\,312\,006\,913 & 4\,311\,720\,449 \\
9 & 68\,854\,546\,945 & 68\,854\,448\,641 & 68\,853\,843\,969 \\
10 & 1\,100\,584\,649\,729 & 1\,100\,596\,118\,529 & 1\,100\,585\,936\,897 \\
11 & 17\,600\,762\,873\,857 & 17\,600\,774\,408\,193 & 17\,600\,759\,586\,817\\
\end{array} \]
\smallskip
\caption{Point counts.}
\label{point_counts}
\end{table}

\section{Verification and implementation}
\label{sec:checks}

Our implementation of
the algorithm described in \S\ref{sec:alg} is included as an ancillary file
\verb|count.cpp|.  We double-checked its output very thoroughly:

\setlength \leftmargini {1.5em}
\begin{itemize}
\setlength \itemsep {4pt}
\item For small $m$, we checked the counts over $\F_{2^m}$ using the naive algorithm discussed at the beginning of \S\ref{sec:alg}.

\item We checked the counts up to about $m=9$ with a ``semi-sophisticated'' algorithm that projects from a point rather than a line.

\item We projected from several different lines and got the same counts.

\item After finding the characteristic polynomial one can predict the counts for all $m$.  We checked these up to $m=14$, and even $m=15$ on a computer with much more memory than the first author's laptop.

\item The characteristic polynomial of $\Phi^*$ acting on
\[ H^4_{\et,\prim}(X_{\bar\F_2},\Q_\ell), \]
with no Tate twist, is $4^{22}\chi(t/4)$, and this must have integer coefficients.  But even stronger, we have
\[ H^4_{\et,\prim}(X_{\bar\F_2},\Q_\ell(1)) \cong H^2_{\et,\prim}(F_{\bar\F_2},\Q_\ell), \]
where $F$ is the Fano variety of lines on $X$, so $2^{22}\chi(t/2)$ must
have integer coefficients.  We verified this.

\item We used our program to count points on Elsenhans and Jahnel's cubic \cite[Example~3.15]{ej_duke}, and our numbers agreed with theirs.
\end{itemize}

\smallskip
\noindent We conclude with a few practical comments about our implementation:
\smallskip

\setlength \leftmargini {1.5em}
\begin{itemize}
\setlength \itemsep {4pt}
\item We represented elements of $\F_{2^m}$ as unsigned integers, interpreting the bits as coefficients of a polynomial in $\F_2[x]$ modulo a fixed irreducible polynomial of degree $m$.  Thus addition is given by ``xor'' and multiplication by a well-known algorithm.  

\item We stored multiplication in a lookup table, which sped up the program by an order of magnitude.

\item We also stored division in a lookup table, as well as roots of quadratic and depressed cubic polynomials, which saved us the trouble of writing those algorithms.  This did not start to use an unreasonable amount of memory until $m = 14$.

\item Following \cite[Alg.~15]{ej_deg2} and \cite[\S8]{hvv}, we pre-computed a list of Galois orbit representatives (and orbit sizes) in $\F_{2^m}$, and then touched each Galois orbit of $\P^2$ only once, which sped up the program by a factor of $m$.

\item We did not bother with parallelization, although this problem is ideally suited to it.
\end{itemize}

\bibliographystyle{plain}
\bibliography{apolar}

\newcommand \httpurl [1] {\href{http://#1}{\nolinkurl{#1}}}
\begin{thebibliography}{10}

\bibitem{akr}
T.~Abbott, K.~Kedlaya, and D.~Roe.
\newblock Bounding {P}icard numbers of surfaces using $p$-adic cohomology.
\newblock In {\em Arithmetics, geometry, and coding theory ({AGCT} 2005)},
  volume~21 of {\em S\'emin. Congr.}, pages 125--159. Soc. Math. France, Paris,
  2010.
\newblock Also \href{http://arxiv.org/abs/math/0601508}{math/0601508}.

\bibitem{abbv}
A.~Auel, M.~Bernardara, M.~Bolognesi, and A.~V{\'a}rilly-Alvarado.
\newblock Cubic fourfolds containing a plane and a quintic del {P}ezzo surface.
\newblock {\em Algebr. Geom.}, 1(2):181--193, 2014.
\newblock Also \href{http://arxiv.org/abs/1205.0237}{arXiv:1205.0237}.

\bibitem{bsd}
E.~Bombieri and H.~P.~F. Swinnerton-Dyer.
\newblock On the local zeta function of a cubic threefold.
\newblock {\em Ann. Scuola Norm. Sup. Pisa (3)}, 21:1--29, 1967.

\bibitem{magma}
W.~Bosma, J.~Cannon, and C.~Playoust.
\newblock The {M}agma algebra system. {I}. {T}he user language.
\newblock {\em J. Symbolic Comput.}, 24(3-4):235--265, 1997.
\newblock Computational algebra and number theory (London, 1993).

\bibitem{cs}
F.~Charles and C.~Schnell.
\newblock Notes on absolute {H}odge classes.
\newblock In {\em Hodge theory}, volume~49 of {\em Math. Notes}, pages
  469--530. Princeton Univ. Press, Princeton, NJ, 2014.
\newblock Also \href{http://arxiv.org/abs/1101.3647}{arXiv:1101.3647}.

\bibitem{chk}
E.~Costa, D.~Harvey, and K.~Kedlaya.
\newblock Zeta functions of nondegenerate toric hypersurfaces via controlled
  reduction in $p$-adic cohomology.
\newblock In preparation.

\bibitem{dlr}
O.~Debarre, A.~Laface, and X.~Roulleau.
\newblock Lines on cubic hypersurfaces over finite fields.
\newblock {\em Simons Publication Series}, to appear.
\newblock Also \href{http://arxiv.org/abs/1510.05803}{arXiv:1510.05803}.

\bibitem{df}
D.~Dummit and R.~Foote.
\newblock {\em Abstract algebra}.
\newblock John Wiley \& Sons, Inc., Hoboken, NJ, third edition, 2004.

\bibitem{ej_deg2}
A.-S. Elsenhans and J.~Jahnel.
\newblock K3 surfaces of {P}icard rank one and degree two.
\newblock In {\em Algorithmic number theory}, volume 5011 of {\em Lecture Notes
  in Comput. Sci.}, pages 212--225. Springer, Berlin, 2008.

\bibitem{ej_duke}
A.-S. Elsenhans and J.~Jahnel.
\newblock On the characteristic polynomial of the {F}robenius on \'etale
  cohomology.
\newblock {\em Duke Math. J.}, 164(11):2161--2184, 2015.
\newblock Also \href{http://arxiv.org/abs/1106.3953}{arXiv:1106.3953}.

\bibitem{fulton}
W.~Fulton.
\newblock {\em Intersection theory}, volume~2 of {\em Ergebnisse der Mathematik
  und ihrer Grenzgebiete. 3. Folge.}
\newblock Springer-Verlag, Berlin, second edition, 1998.

\bibitem{M2}
D.~Grayson and M.~Stillman.
\newblock Macaulay2, a software system for research in algebraic geometry.
\newblock Available at \httpurl{www.math.uiuc.edu/Macaulay2/}.

\bibitem{harvey}
D.~Harvey.
\newblock Computing zeta functions of arithmetic schemes.
\newblock {\em Proc. Lond. Math. Soc. (3)}, 111(6):1379--1401, 2015.
\newblock Also \href{http://arxiv.org/abs/1402.3439}{arXiv:1402.3439}.

\bibitem{hassett}
B.~Hassett.
\newblock Special cubic fourfolds.
\newblock {\em Compositio Math.}, 120(1):1--23, 2000.
\newblock Also \httpurl{www.math.brown.edu/~bhassett/papers/cubics/cubic.pdf}.

\bibitem{hassett_survey}
B.~Hassett.
\newblock Cubic fourfolds, {K}3 surfaces, and rationality questions.
\newblock In {\em Rationality problems in algebraic geometry}, volume 2172 of
  {\em Lecture Notes in Math.}, pages 29--66. Springer, Cham, 2016.
\newblock Also \href{http://arxiv.org/abs/1601.05501}{arXiv:1601.05501}.

\bibitem{hvv}
B.~Hassett, A.~V{\'a}rilly-Alvarado, and P.~Varilly.
\newblock Transcendental obstructions to weak approximation on general {K}3
  surfaces.
\newblock {\em Adv. Math.}, 228(3):1377--1404, 2011.
\newblock Also \href{http://arxiv.org/abs/1005.1879}{arXiv:1005.1879}.

\bibitem{mukai}
S.~Mukai.
\newblock Fano 3-folds.
\newblock In {\em Complex projective geometry ({T}rieste, 1989/{B}ergen,
  1989)}, volume 179 of {\em London Math. Soc. Lecture Note Ser.}, pages
  255--263. Cambridge Univ. Press, Cambridge, 1992.
\newblock Also \httpurl{www.kurims.kyoto-u.ac.jp/~mukai/paper/Trieste.pdf}.

\bibitem{murre}
J.~P. Murre.
\newblock On the {H}odge conjecture for unirational fourfolds.
\newblock {\em Indagationes Math. (Proc.)}, 80(3):230--232, 1977.

\bibitem{rs}
K.~Ranestad and F.-O. Schreyer.
\newblock Varieties of sums of powers.
\newblock {\em J. Reine Angew. Math.}, 525:147--181, 2000.
\newblock Also \href{http://arxiv.org/abs/math/9801110}{math/9801110}.

\bibitem{rv}
K.~Ranestad and C.~Voisin.
\newblock Variety of power sums and divisors in the moduli space of cubic
  fourfolds.
\newblock {\em Doc. Math.}, 22:455--504, 2017.
\newblock Also \href{http://arxiv.org/abs/1309.1899}{arXiv:1309.1899}.

\bibitem{terasoma}
T.~Terasoma.
\newblock Complete intersections with middle {P}icard number 1 defined over
  {${\bf Q}$}.
\newblock {\em Math. Z.}, 189(2):289--296, 1985.

\bibitem{vl_older}
R.~van Luijk.
\newblock An elliptic {K3} surface associated to {H}eron triangles.
\newblock {\em J. Number Theory}, 123(1):92--119, 2007.
\newblock Also \href{http://arxiv.org/abs/math/0411606}{math/0411606}.

\bibitem{vl}
R.~van Luijk.
\newblock K3 surfaces with {P}icard number one and infinitely many rational
  points.
\newblock {\em Algebra Number Theory}, 1(1):1--15, 2007.
\newblock Also \href{http://arxiv.org/abs/math/0506416}{math/0506416}.

\bibitem{voisin_thesis}
C.~Voisin.
\newblock Th\'eor\`eme de {T}orelli pour les cubiques de {$\P^5$}.
\newblock {\em Invent. Math.}, 86(3):577--601, 1986.
\newblock Also
  \httpurl{webusers.imj-prg.fr/~claire.voisin/Articlesweb/torelli.pdf}.

\bibitem{voisin_Z_hodge}
C.~Voisin.
\newblock Some aspects of the {H}odge conjecture.
\newblock {\em Jpn. J. Math.}, 2(2):261--296, 2007.
\newblock Also
  \httpurl{webusers.imj-prg.fr/~claire.voisin/Articlesweb/takagifinal.pdf}.

\bibitem{zucker}
S.~Zucker.
\newblock The {H}odge conjecture for cubic fourfolds.
\newblock {\em Compositio Math.}, 34(2):199--209, 1977.

\end{thebibliography}

\end{document}